\documentclass[11pt]{amsart}
\usepackage{amsmath,amsthm,amssymb,amscd,url}
\usepackage{verbatim}
\usepackage[all]{xy} \SelectTips {cm}{}
\usepackage{tikz}
\usetikzlibrary{matrix,arrows,decorations.pathreplacing,decorations.pathmorphing}
\usepackage{appendix}
\usepackage{mathrsfs}
\usepackage[margin=1in]{geometry}

\usepackage{hyperref}
\hypersetup{
    colorlinks,
    citecolor=blue,
    filecolor=blue,
    linkcolor=blue,
    urlcolor=blue
}

\newtheorem{theorem}{Theorem}[section] 
\newtheorem{lemma}[theorem]{Lemma}     
\newtheorem{corollary}[theorem]{Corollary}

\newtheorem{proposition}[theorem]{Proposition}

\theoremstyle{remark}
\newtheorem{remark}[theorem]{Remark}
\theoremstyle{definition}

\usepackage{amsfonts,url}

\def\Irr{mathbf{Irr}}
\def\fB{\mathfrak{B}}
\def\Irr{\mathrm{Irr}}

\def\Spec{\mathrm{Spec}}

\def\End{\mathrm{End}}
\def\Ind{\mathrm{Ind}}

\def\fo{\mathfrak{o}}
\def\fX{\mathfrak{X}}

\newcommand{\Z}{\mathbb{Z}}
\newcommand{\R}{\mathbb{R}}
\newcommand{\C}{\mathbb{C}}

\newcommand{\cH}{\mathcal{H}}

\newcommand{\cO}{\mathcal{O}}

\def\cT{\mathcal{T}}
\def\cK{\mathcal{K}}

\newcommand{\SL}{\mathrm{SL}}

\newcommand{\GL}{\mathrm{GL}}

\newcommand{\nr}{\mathrm{nr}}
\newcommand{\unr}{\mathrm{unr}}

\def\temp{\mathrm{t}}
\def\val{\mathrm{val}}

\newcommand{\fs}{\mathfrak{s}}
\def\fA{\mathfrak{A}}
\def\fC{\mathfrak{C}}
\def\fB{\mathfrak{B}}
\def\fp{\mathfrak{p}}
\def\fK{\mathfrak{K}}

\def\ft{\mathfrak{t}}

\begin{document}
\title[Klein bottle]{A twisted Hecke algebra, then and now, and a Klein bottle of tempered representations}

\begin{abstract} Let $F$ be a non-archimedean local field such that $4|q-1$, with $q$ the order of the residue field of $F$, and let $(M^0,\sigma^0)$ be the depth-zero
cuspidal pair for the  twisted Levi subgroup $G^0$ of $\SL_8$ arising from quadratic and quartic field extensions, as defined in \cite{AFO}.   
Then the corresponding Bernstein block is described by a twisted Hecke algebra $\cH^0$.   We describe $\cH^0$ explicitly as a noncommutative $\C$-algebra with generators and relations.   
We describe explicitly the simple modules of $\cH^0$.   All the simple modules are $2$-dimensional.   The  primitive spectrum of $\cH^0$ is then an explicit  complex algebraic variety $\fX$.
 The maximal compact real form of $\fX$ is homeomorphic to a Klein bottle. This Klein bottle is a model of the unitary principal series of $G^0$ attached  to the cuspidal pair $(M^0, \sigma^0)$.  
 
In addition, we make a full comparison with the classical situation in which 
$G = \SL_8$ and $(M,\sigma)$ is a cuspidal pair for $G$.   The supercuspidal representantion $\sigma$ is constructed from the same quadratic and quartic extensions of $F$.   Let $\fs$ be the point in the Bernstein spectrum $\fB G$ determined by $(M,\sigma)$.  
We compare the two points $\fs$ and $\fs^0$ and show explicitly that the corresponding Bernstein varieties 
  are isomorphic.   In that case, the Klein bottle re-appears, this time floating in the tempered dual of $\SL_8$.   
 \end{abstract}

\author[A.-M. Aubert]{Anne-Marie Aubert}
\address{Sorbonne Universit\'e and Universit\'e Paris Cit\'e, CNRS, IMJ-PRG, F-75005 Paris, France}
\email{anne-marie.aubert@imj-prg.fr}
\author[R. Plymen]{Roger Plymen}\address{
School of Mathematics, Southampton University, Southampton SO17 1BJ,  England 
\emph{and} School of Mathematics, Manchester University, Manchester M13 9PL, England}\email{
r.j.plymen@soton.ac.uk \quad roger.j.plymen@manchester.ac.uk}

\date{\today}
%\subjclass[2010]{20G25, 22E50}
%\keywords{representation theory, division algebra, Hecke algebra, types}
\maketitle

\tableofcontents

\section{Introduction}  

Let $G$  denote a connected reductive group over a non-archimedean local field $F$.  Then, following \cite{Ber}, the category $\mathrm{Rep}(G(F))$ of smooth, complex representations of $G(F)$ is a direct
product of full subcategories called ``Bernstein blocks":
\[
\mathrm{Rep}(G(F)) = \prod_{\fs \in \mathfrak{B}G} \mathrm{Rep}^\fs(G(F)).
\]
Each of the blocks $\mathrm{Rep}^\fs(G(F))$ is equivalent to the category of unital right modules
over an algebra $\cH^\fs$.   Suppose that the category $\mathrm{Rep}^\fs(G(F))$ has an associated type,
as defined by Bushnell and Kutzko \cite{BK}, i.e., a compact open subgroup $K$ of $G(F)$
and an irreducible smooth representation $\rho$ of $K$ such that a representation $\pi \in 
\mathrm{Rep}(G(F))$ belongs to $\mathrm{Rep}^\fs(G(F))$ if and only if every irreducible subquotient of $\pi$ 
contains $\rho$ upon restriction to $K$. Then we can replace the algebra $\cH^\fs$ by the Hecke
algebra $\cH(G(F),(K, \rho))$ of all compactly supported, $\End_{\C}(\rho)$-valued functions on
$G(F)$ that transform on the left and right according to $\rho$. That is, $\mathrm{Rep}^\fs(G(F))$ is
equivalent to the category of modules over $\cH(G(F),(K,\rho))$.

We shall apply these considerations to the  group $G^0$, a twisted Levi of $\SL_8$, recently considered in Adler-Fintzen-Ohara \cite{AFO}.   Following their construction, let $M^0$ be a maximal torus in $G^0$ and 
let $\sigma^0$ be the quadratic character  of $M^0$ defined in their article.   Let $\fs^0$ be the point in the Bernstein spectrum $\mathfrak{B}G^0$ determined by the cuspidal pair $(M^0, \sigma^0)$.
  Let $(K^0, \rho^0)$ be the Iwahori pair constructed in \cite{AFO}.   Then $(K^0, \rho^0)$ is an $\fs^0$-type.   

We will write
\[
\cH^0 = \cH^{\fs^0} = \mathcal{H}(G^0(F), (K^0, \rho^0)).
\]
In particular, $\cH^0$ is an example of
a Hecke algebra attached to a depth-zero, principal-series block of a quasi-split
group that requires a non-trivial $2$-cocycle.   In the context of the present article, we have
\[
\boxed{
\mathcal{H}(G^0(F), (K^0, \rho^0)) \simeq  \C[W_{\mathrm{aff}}(\widetilde{A}_1) \times \Z, \mu^\mathcal{T}].
}
\]
The group $W_{\mathrm{aff}}(\widetilde{A}_1)$ is the affine Weyl group of the affine root system of type $\widetilde{A}_1$ and  $\mu^\mathcal{T}$ is a non-trivial $2$-cocyle (see \cite[Corollary~3.9]{AFO}). It is defined as the restriction to $(W_{\mathrm{aff}}(\widetilde{A}_1) \times \Z)^2$  of the $2$-cocyle constructed in \cite[Notation~3.6.1]{AFMO1}, which is associated to a family $\mathcal{T}$ of intertwining operators.

In section 2 we review the notation and definitions from \cite{AFO} and check that the Iwahori pair $(K^0, \rho^0)$ is an $\fs^0$-type.   

In section 3 we compute explicitly the Bernstein torus $T^{\fs^0}$, the  group $W^{\fs^0}$ and the variety $T^{\fs^0} / W^{\fs^0}$.   We note  that $T^{\fs^0}$  is  a complex torus of dimension $2$.
 We write 
\[
\mathfrak{X} = T^{\fs^0} / W^{\fs^0}
\]
and prove that the maximal compact real form of $\fX$ is  a Klein bottle $\mathbf{Kb}$.   This Klein bottle $\mathbf{Kb}$ is a model of the unitary principal series of $G^0$ attached  to the cuspidal pair $(M^0, \sigma^0)$.   

Section 4 is devoted to background material in group cohomology.   

In section 5 we provide a  description of the twisted group algebra $\C[W_{\mathrm{aff}}(\widetilde{A}_1) \times \Z, \mu^\mathcal{T}]$ as a 
noncommutative $\C$-algebra with explicit generators and relations.  

In section 6 we compute the simple modules of the twisted group algebra.   We prove that their parameter space is, as expected, isomorphic to $\mathfrak{X}$, an irreducible complex algebraic variety.      All the simple modules are $2$-dimensional.   

In section 7 we compute, by way of comparison, the simple modules of the untwisted group algebra.   The untwisted group algebra admits $1$-dimensional simple modules as well as $2$-dimensional simple modules.   

In section 8 we show explicitly how the Klein bottle arises as a compact Hausdorff subspace of the tempered dual of $G^0$.   

In section 9 we make a full comparison with the classical example in Roche \cite{R} and Goldberg-Roche \cite{GR2}.   This classical example is due ultimately to Phil Kutzko.   
In this section we write $G = \SL_8(F)$ and, as in \cite{R}, we only assume that $4|(q-1)$, where $q$ is the order of the residue field of $F$, a finite power of the residual characteristic $p$ of $F$.
We observe that this assumption allows us to consider small odd primes $p$, that is, cases where $p$ divides the order (which equals $8!$) of the Weyl group of $G$, e.g., $p=3$ (with $q=p^2=9$),  $p=5$ (with $q=p$ or $q=p^2=25$), and  $p=7$ (with $q=p$).

Then $G$ admits a cuspidal pair $(M,\sigma)$ where $\sigma$ is constructed via the same quadratic and quartic extensions of $F$.  
 Let $\fs$ be the point in the Bernstein spectrum $\fB G$ determined by the cuspidal pair $(M, \sigma)$.   We compare the two points $\fs$ and $\fs^0$ and show explicitly that the corresponding Bernstein varieties 
  are isomorphic.   In that case, the Klein bottle re-appears, this time floating in the tempered dual of $\SL_8$.

\subsubsection{Geometric meaning of the cocycle.}

The group $W_{\mathrm{aff}}(\widetilde{A}_1)$ is the infinite dihedral group
\[\mathrm{D}_\infty:=\left\langle s,t\,:\, s^2=t^2=1\right\rangle\,\simeq\,\Z/2\ast\Z/2,\]
which also admits the following presentation
\[\mathrm{D}_\infty=\left\langle r,s\,:\, s^2=1, srs=r^{-1}\right\rangle\,\simeq\,\Z \rtimes \Z/2.\]
 It is instructive to compare the untwisted group algebra $\C[\Gamma]$, where $\Gamma : = (\Z \rtimes \Z/2) \times \Z$,
with the twisted algebra $\C[\Gamma,\mu]$ with $\mu\colon \Gamma\times \Gamma\to\C^\times$ a non-trivial $2$-cocycle.

In the absence of the cocycle, the induced $\Z/2$--action on the
character torus $(\C^\times)^2$ is
\[
(w,z)\longmapsto (w,z^{-1}),
\]
whose fixed points give rise to one--dimensional simple modules.
After twisting, the commutation relation between the $\Z/2$--generator
and the $\Z$--direction is modified, and the involution becomes
\[
(w,z)\longmapsto (-w,z^{-1}),
\]
which acts freely on $(\C^\times)^2$.
Consequently all simple modules become two--dimensional, and on the
maximal compact real form $S^1\times S^1$ the resulting free
orientation-reversing quotient is a Klein bottle.
In this sense, the Klein bottle appearing in the tempered dual is the
topological manifestation of the nontrivial cohomology class
$[\mu]\in H^2(\Z/2\times\Z,\C^\times)$.

\smallskip

In summary, we have

\begin{theorem}  Let  $(M^0,\sigma^0)$ be the depth-zero
cuspidal pair for $G^0$ arising from quadratic and quartic field extensions of $F$.
Then the  Bernstein block $\mathrm{Rep}^{\fs^0}(G^0(F))$ is described by a twisted Hecke algebra whose primitive spectrum is the irreducible complex algebraic variety $\fX = (\C^\times  \times \C^\times) / (w,z) \sim (-w, z^{-1})$.
Its maximal compact real form is a Klein bottle. This Klein bottle is a model of the unitary principal series of $G^0$ attached  to the cuspidal pair $(M^0, \sigma^0)$.  

Let $G=\SL_8$ and let $\fs$ be the point in the Bernstein spectrum $\fB G$ determined by $(M,\sigma)$.  
The Bernstein varieties associated to the points $\fs$ and $\fs^0$ are isomorphic.   In that case, the Klein bottle re-appears, this time floating in the tempered dual of $\SL_8$.   
\end{theorem}

In both cases of $G^0$ and $G$, a \emph{non-trivial} cocycle in the Hecke algebra is associated with a compact connected \emph{non-orientable} component in the tempered dual.   
We are tempted to ask whether this phenomenon arises in other examples.

\section{Depth zero}  We begin by recalling the notation in \cite{AFO}.    Let $F$ denote a non-archimedean local field with residue field $\mathfrak{f}$  of characteristic $p$   
(assumed odd) and order $q$.  Let $\fo_F$  denote the ring of  integers of $F$, and $\fp_F$ the maximal ideal of $\fo_F$.  We fix a uniformizer $\varpi_F$ in $F$ and let $\zeta$  denote a primitive $(q-1)$-st root of unity in $F$.   
 We assume also that $4$ divides $q-1$.   It follows that there is a unique character $\eta \colon F^\times \to \C^\times$ that is trivial on $\varpi_F$ and $1 + \fp_F$ and satisfies $\eta(\zeta) = \sqrt{-1}$.   
 
 Let $\chi_0$ be the unramified character of $F^\times$ such that $\chi_0(\varpi_F ) = - \sqrt{-1}$.   
 
 Let $E_2$ be the splitting field of the polynomial $X^2 + \varpi_F$ 
 and  $E_4$ be the splitting field of the polynomial $X^4 + \zeta \varpi_F$.   Note that the
fields that we denote by $E_2$  and $E_4$  here are denoted by $E_1$ and $E_2$, respectively,
in \cite{R}. Let $\varpi_{E_2}$, resp. $\varpi_{E_4}$, denote a uniformizer of $E_2$, resp., $E_4$, such that
$\varpi^2_{E_2} = - \varpi_F$, resp. $\varpi^4_{E_4} = - \zeta \varpi_F$.   

Let
\begin{align*}
G^0 &= (\GL_2(E_2) \times \GL_1(E_4)) \cap \SL_8(F)\\
K^0 & = G^0(F)_{x_0} =  ( I_2  \times I_4) \cap \SL_8(F)\\
\rho^0\colon K^0 & \to  \C^\times, \quad \quad ((\begin{smallmatrix}
a&b\\
c & d
\end{smallmatrix}),-) \mapsto ( \eta \circ N_{E_2/F})(d)\\
\eta \colon F^\times & \to  \C^\times  \quad  \textrm{is the depth-zero character from \S1}\\
M^0 &= (\GL_1(E_2) \times \GL_1(E_2) \times \GL_1(E_4)) \cap \SL_8(F)\\
\sigma^0\colon M^0 & \to  \C^\times, \quad \quad 1 \otimes \eta \circ N_{E_2/F} \otimes 1 
%K_{M^0} &=& (O^\times_{E_2} \times O^\times_{E_2} \times O^\times_{E_4}) \cap \SL_8(F) \quad \textrm{if} \quad K^0 = G^0(F)_{x_0}\\
\end{align*}
with notation (apart from our definition of $\sigma^0$) from \cite{AFO}.   Then we have an isomorphism of noncommutative $\C$-algebras:
\[
\boxed{
\mathcal{H}(G^0(F), (K^0, \rho^0)) \simeq  \C[W_{\mathrm{aff}}(\widetilde{A}_1) \times \Z, \mu^\mathcal{T}]
}
\]
by \cite[Prop. 3.4 and  Corollary 3.9]{AFO}.

%We have $W(G^0, M^0)= \Z/2\Z$.

%\subsubsection{On the point $\fs^0$ in the Bernstein spectrum $\mathcal{B}G^0$}  We recall the definitions of $G^0, K^0$ etc. 

\begin{lemma}  Let $\fs^0$ be the point in the Bernstein spectrum $\mathfrak{B}G^0$ determined by the  cuspidal pair $(M^0, \sigma^0)$.   Then $(K^0, \rho^0)$ is an $\fs^0$-type.  
\end{lemma}

 \begin{proof}
Let $P^0=M^0N^0$ be the parabolic subgroup of $G^0$ with Levi factor $M^0$ defining the inertial
class $\fs^0=[M^0,\sigma^0]$.  By construction,
\[
K^0=(I_2\times I_4)\cap \SL_8(F)
\]
admits an Iwahori decomposition with respect to $P^0$:
\[
K^0=(K^0\cap \overline{N^0})\,(K^0\cap M^0)\,(K^0\cap N^0).
\]
The representation $\rho^0$ is a depth-zero character, hence is trivial on the pro-$p$ radicals.
In particular it is trivial on the unipotent parts $K^0\cap N^0$ and $K^0\cap \overline{N^0}$.
Moreover, by the defining formulae, the restriction of $\rho^0$ to $K^0\cap M^0$ coincides with
$\sigma^0|_{K^0\cap M^0}$.

It follows from \cite[\S 4.9]{Mo} that $(K^0,\rho^0)$ is a $G^0$-cover of the pair \[(K^0\cap M^0,\;\sigma^0|_{K^0\cap M^0}).\] 
and hence is an $\fs^0$-type for $G^0$.\end{proof}

\section{On the depth-zero variety}

We will begin by collecting some useful results.

\begin{lemma}\label{etas}
\begin{align}
(\eta \circ N_{E_2/F}) \otimes 1 \otimes 1|_{M^0} &= 1 \otimes (\eta^{-1} \circ N_{E_2/F})  \otimes (\eta^{-1} \circ N_{E_4/F})|_{M^0}\\
(\eta^2 \circ N_{E_2/F})(E_2^\times) & = 1\\
(x,y,z) \in M^0  \implies (\eta^2 \circ N_{E_4/F}) (z) &= 1\\
(\eta^{-1} \circ N_{E_4/F})(\mathcal{O}^\times_{E_4}) & = 1. 
\end{align}
\end{lemma}

\begin{proof}    We recall that 
\begin{align*}
M^0 &= \{(x,y,z) \in E_2^\times \times  E_2^\times \times E_4^\times \,:\,  N_{E_2/F}(xy) \cdot N_{E_4/F}(z) = 1\}\\
\sigma^0(x,y,z) & =( \eta \circ N_{E_2/F})(y).
\end{align*}

If $(x,y,z) \in M^0$ then we have the norm-one condition
\begin{align}\label{NE2}
N_{E_2/F}(x) \cdot N_{E_2/F}(y) \cdot  N_{E_4/F}(z) = 1.
\end{align}

(1) follows by applying $\eta$ to this equation.   

(2): we have, from \cite[\S4.1]{R}, the following equation:
\begin{align*}
N_{E_2/F} (E_2^\times) &= (1 + \fp_F) \langle \zeta^2 \rangle \langle \varpi_F \rangle. 
\end{align*}
Therefore, we have
\begin{align*}
 \eta^2 (N_{E_2/F} E_2^\times) &= \eta^2(\langle \zeta^2 \rangle ) \\
& = 1.
\end{align*}
It follows that $\sigma^0$ is a quadratic chracter of $M^0$.     

(3) follows from (2) and (5) after applying $\eta^2$.   

(4): we proceed as follows.     In this paragraph, we write $E = E_4$.   Every $z\in \fo^\times_{E}$ can be written as
\[
z = u\cdot (1+\alpha),
\qquad   u\in\mu_{q-1},\ \alpha\in\mathfrak p_E.
\]
Since $E/F$ is tame, one has $N_{E/F}(1+\mathfrak p_E)=1+\mathfrak p_F$; thus $\eta$ (hence also $\eta^{-1}$) is trivial on norms of principal units.
Also, because $E/F$ is totally ramified, $k_{E} =k_F$, so the Teichm\"uller subgroup $\mu_{q-1}\subset \fo_E^\times$ is identified with that in $F$.
For $u\in\mu_{q-1}$, the Galois group acts trivially on $u$, so
\[
N_{E/F}(u)=u^{[E:F]}=u^4,
\]
and $\eta(u^4)=1$ because $\eta(\zeta^4)=\eta(\zeta)^4=i^4=1$ (hence $\eta$ kills all fourth powers in $\mu_{q-1}$).
Consequently,
\[
\eta^{-1}\!\bigl(N_{E/F}(u(1+\alpha))\bigr)=1
\quad\text{for all }u\in\mu_{q-1},\ \alpha\in\mathfrak p_E.
\]
\end{proof}

\begin{theorem}\label{ETAS}  We have
\[
{}^{\widetilde{s}} \sigma^0 \cong \chi \cdot \sigma^0
\]
where $\chi$ is an unramified character of $M^0$.   
\end{theorem}

\begin{proof}   We will adapt the proof of Lemma 3.1 in \cite{AFO} to the present situation.

%We need an explicit formula for the character $\eta^{-1} \circ N_{E_4/F} : E_4^\times \to \C^\times$.     Lemma \ref{NE4} provides an explicit formula. 
Define
\[
\chi_4 \colon E_4^\times \to \C^\times, \qquad z \mapsto (\eta^{-1} \circ N_{E_4/F}(z).
\]
When we apply successively (\ref{etas}) and (2), we obtain
\begin{align*}
{}^{\widetilde{s}} \sigma^0 & = {}^{\widetilde{s}}(1 \otimes (\eta \circ N_{E_2/F}) \otimes 1)|_{M^0}\\
& = \left((\eta \circ N_{E_2/F}) \otimes 1 \otimes 1\right)|_{M^0}\\
& = \left(1 \otimes (\eta^{-1} \circ N_{E_2/F})  \otimes (\eta^{-1} \circ N_{E_4/F}) \right)|_{M^0}\\
& = \left(1 \otimes (\eta \circ N_{E_2/F})  \otimes \chi_4) \right)|_{M^0}\\
& = \chi \sigma^0
\end{align*}
where $\chi = 1 \otimes 1 \otimes \chi_4$.  The character $\chi$ is trivial on the maximal compact subgroup $K_{M^0}$ of $M^0$ and so $\chi$ is an \emph{unramified} character of $M^0$. 
\end{proof}  

We recall the definition of $W^{\fs^0}$:
\[W^{\fs^0}:=\left\{n\in N_{G^0}(M^0)\,: \,{}^n\sigma^0\simeq \sigma^0\chi \text{ for some unramified character $\chi$ of $M^0$}\right\}.\]

\begin{corollary}   Let $\fs^0$ denote the point in the Berstein spectrum $\mathfrak{B}G^0$ determined by the  cuspidal pair $(M^0, \sigma^0)$.     Then we have
\[
W^{\fs^0} = \{1, \widetilde{s}\}.
\]
\end{corollary}

\begin{lemma}\label{normE4}
Let $E_4/F$ be as in \S2, with chosen uniformizers $\varpi_F\in F$ and
$\varpi_{E_4}\in E_4$ satisfying $\varpi_{E_4}^4=-\zeta\,\varpi_F$, and let
$\eta:F^\times\to\C^\times$ be the depth-zero character with $\eta(\varpi_F)=1$
and $\eta(\zeta)=\sqrt{-1}=i$.

(1) If $z \in E_4^\times$ then
\[
\eta^{-1}\!\left(N_{E_4/F}(z)\right)=(-i)^{\,v_{E_4}(z)}.
\]

(2) If $(x,y,z) \in M^0$ and $v_{E_4}$ denotes normalised valuation on $E_4^\times$ then $v_{E_4}(z) \equiv  0 \pmod 2$.   
\end{lemma}

\begin{proof} We prove (1) first.   Write $z=\varpi_{E_4}^{\,k}u$ with $k=v_{E_4}(x)\in\Z$ and $u\in\cO_{E_4}^\times$.  
Then we have
\[
\eta^{-1}\!\left(N_{E_4/F}(u)\right)=1
\]
by Lemma \ref{etas} part (4).   

Next, because $\mu_4\subset F^\times$, the $F$-conjugates of $\varpi_{E_4}$ are
$\varpi_{E_4},\,i\varpi_{E_4},\,i^2\varpi_{E_4},\,i^3\varpi_{E_4}$, so
\[
N_{E_4/F}(\varpi_{E_4})
=\prod_{j=0}^3 i^j\,\varpi_{E_4}
=i^{0+1+2+3}\,\varpi_{E_4}^4
=i^6\,\varpi_{E_4}^4
=i^2\,\varpi_{E_4}^4
=(-1)\,\varpi_{E_4}^4.
\]
Using $\varpi_{E_4}^4=-\zeta\,\varpi_F$ we obtain $N_{E_4/F}(\varpi_{E_4})=\zeta\,\varpi_F$,
hence
\[
\eta^{-1}\!\left(N_{E_4/F}(\varpi_{E_4})\right)=\eta^{-1}(\zeta)\eta(\varpi_F)= - i.
\]
Consequently,
\[
\eta^{-1}\!\left(N_{E_4/F}(z)\right)
=\eta^{-1}\!\left(N_{E_4/F}(\varpi_{E_4})\right)^k\,\eta^{-1}\!\left(N_{E_4/F}(u)\right)
= (-i)^k.
\]

We now prove (2).   If $(x,y,z) \in M^0$ then $\eta^2(N_{E_4/F}(z))=1$ by   Lemma \ref{etas} part (3). This gives $(- i)^{2k} =  (-1)^k  = 1$, so $k$ is even.
\end{proof}

\subsubsection{Remark}    We note that the following two statements are consistent:
\begin{itemize}
\item  The character $\chi_4 \colon E_4^\times \to \C^\times$ \quad \textrm{is of order 4}
\item The character $(1 \otimes 1 \otimes \chi_4)_{M^0} \colon M^0 \to \C^\times$ \quad \textrm{is of order 2}.
\end{itemize} 

 We recall the definition of $\chi_4$ in Theorem \ref{ETAS}:
\[
\chi_4: = \eta^{-1} \circ N_{E_4/F}.
\]

For the normalised valuations $\val_F$ and $v_{E_4}$, we have 
\[
\val_F(N_{E_4/F}(x)) = f(E_4/F)v_{E_4}(x)
\]
for all $x \in E_4$.   Since $E_4/F$ is totally ramified, we have $f(E_4/F) = 1$.   So we have
\[
v_{E_4} = \val_F \circ N_{E_4/F}.
\]

We can now proceed as follows.  We have,  in the notation of Theorem \ref{ETAS}, using Lemma \ref{normE4}:
\begin{align*}
(x,y,z) & \in M^0 \\
\implies \chi(x,y,z) &= (1 \otimes 1 \otimes \chi_4)(x,y,z)\\
 &=  \chi_4(z)\\
&= \eta^{-1}(N_{E_4/F}(z))\\
&= (-i)^{v_{E_4}(z)}\\
& = ( - i)^{\val_F(N_{E_4/F}(z))}
\end{align*}
 
Therefore, in the coordinates presently to be defined, the twisting by $\chi$ which occurs in Theorem\ref{ETAS} 
is the twisting by the unramified character $(1 : 1: -i)$.

\subsubsection{The Bernstein torus $T^{\fs^0}$}   Let $X_{\nr}(M^0)$ denote the group of unramified characters of $M^0$.   We recall the definition of the Bernstein torus attached to the cuspidal pair
$(M^0, \sigma^0)$:
\[
T^{\fs^0} = \{\psi \otimes \sigma \,:\, \psi \in X_{\nr}(M^0)\}.   
\]

Let $X_{\nr}(E_2^{\times})$ (resp. $X_{\nr}(E_4^\times)$) denote the group of unramified characters of $E_4^\times$ (resp. $E_4^\times$).   Then we have
\[
T^{\fs^0} = \{\psi_1 \otimes \psi_2 ( \eta \circ N_{E_2/F}) \otimes \psi_3\colon \psi_1, \psi_2  \in X_{\nr}(E_2^\times), \psi_3 \in X_{\nr}(E_4^\times) \}.
\]

We will assign coordinates $a_1, a_2, a_3 \in \C^\times$ as follows:
\begin{align*}
\psi_1(x) &= a_1^{\val_F (N_{E_2/F}(x))}\\
\psi_2(y) &= a_2^{\val_F (N_{E_2/F}(y))}\\
\psi_3(z) &= a_3^{\val_F (N_{E_4/F}(z))}
\end{align*}
and write 
\[
\psi = \psi_1 \otimes \psi_2 \otimes \psi_3.
\]

  We recall the norm $1$ equation (5) for elements $(x,y,z) \in E_2^\times \times E_2^\times \times E_4^\times$:
\[
N_{E_2/F}(x) N_{E_2/F}(y) N_{E_4/F}(z) = 1.
\]
It follows that 
\begin{align}\label{VAL}
\val_F (N_{E_2/F}(x)) + \val_F( N_{E_2/F}(y)) + \val_F( N_{E_4/F}(z) ) &= 0.
\end{align}

%Now consider the following formula\begin{align*}\psi(x_1, x_2, x_3) &= a_1^{\val(N_{E_2/F}(x_1))} a_2^{\val(N_{E_2/F}(x_2))} a_3^{\val(N_{E_4/F})(x_3))} \\(a_1, a_2, a_3) & \in (C^\times)^3.\end{align*}

If $(x, y, z) \in \fo_{E_2}^\times \times \fo_{E_2}^\times \times \fo_{E_4}^\times$ then $\psi(x, y, z) = 1$ and so $\psi$ is indeed unramified.   To be sure of this,
let $E$ be $E_2$ or $E_4$.   Then $x \in \fo_E^\times \implies v_E(x) = 0 \implies \val_F(N_{E/F}(x)) = f(E/F)v_E(x) = 0$.   
Note that, if $a_1 = a_2 = a_3$ then $\psi(x_1, x_2, x_3) = 1$ by (\ref{VAL}).  
 
  We have, in effect, homogeneous coordinates so we will write 
\[
\psi = (a_1: a_2 : a_3 ).
\]
%We will set $a_3 = 1$ in which case $X_{\nr}(M^0$ is parametrized by two nonzero complex numbers $a_1, a_2$ and \[T^{\fs^0} = X_{\nr}(M^0) \cong (\C^\times)^2\]
%a complex torus of dimension $2$.   If $a_1, a_2$ have modulus $1$ then we have the compact torus $S^1 \times S^1$.   

We wish to compute the quotient variety
\[
T^{\fs^0} / W^{\fs^0}.
\]

 Suppose that $\sigma^0$ is twisted by an unramified character $\psi$ with coordinates $(a_1: a_2 : a_3)$. Then 
 \[
 {}^{\widetilde{s}}(\psi_1 \otimes \psi_2 ( \eta \circ N_{E_2/F}) \otimes \psi_3) = \psi_2 \otimes \psi_1 ( \eta \circ N_{E_2/F}) \otimes \chi_4 \psi_3)
 \] 
 as in the proof of Theorem \ref{ETAS}.   In terms of coordinates, this is
 \begin{align}\label{s-action}
s \cdot (a_1: a_2 : a_3) &= (a_2 : a_1: -ia_3)\\
& = (ia_2: ia_1: a_3)
\end{align}

Let the $2$-element group $J = \{1, -1\}$ act on $(\C^\times)^2$ as follows:
 \[
 (-1)(b,c) = (-b, -c)
 \]
 and let $(\C^\times)^2 / J $ denote the quotient .   Define
 \[
 s \cdot (b,c) = (ic,ib)
 \]
 where $s$ is the generator of $W^{\fs^0}$.

\begin{lemma}    Let $T^{\fs^0}$ be the Bernstein torus attached to the point $\fs^0$ in the Bernstein spectrum $\mathfrak{B}G^0$.   Then we have a $W^{\fs^0}$-equivariant isomorphism
\[
  T^{\fs^0} \cong (\C^\times)^2 / J, \qquad (a_1: a_2: a_3) \mapsto (a_1/a_3, a_2/a_3).   
\]
\end{lemma}
\begin{proof}   We shall work with the inverse of this map.   Consider
\begin{align*} f \colon (\C^\times)^3 & \to T^{\fs^0} \\
(a,b,c)  & \mapsto  a^{\val(N_{E_2/F}(x_1))} b^{\val(N_{E_2/F}(x_2))} c^{\val(N_{E_4/F}(x_3))}.
\end{align*}

Note that 
\[
f(a,a,a) = 1
\]
for all $a \in \C^\times$ by (\ref{VAL}).    
  We note that $z$ has even valuation by Lemma \ref{normE4} so $z = \varpi_E^{2m} u$.   Therefore 
\[
\val_F(N_{E_4/F}(z)) =  \val_F(\zeta^{2m} \varpi_F^{2m})  = 2m
\]
and so   
\[
(-1)^{\val(N_{E_4/F}(z))} = 1.
\]
It follows that 
\[
f(b,b, -b) = 1
\]
for all $b \in C^\times$. These two families generate the diagonal $\C^\times$-subgroup and the sign subgroup $J$ hence the kernel is precisely their product. The kernel of $f$ is as required.   
Finally, by (\ref{s-action}), the generator $s$ of $W^{\fs^0}$ acts on the projective coordinates in $(\C^\times)^3 / \C^\times$ as follows:
\[
s \cdot (a : b : c) = (b : a : -  ic) = (ib :ia : c).
\]
\end{proof}

\begin{theorem}[Topology of the depth--zero quotient]\label{prop:Klein}   For the  Bernstein variety, we have $T^{\fs^0} / W^{\fs^0}$ is a complex algebraic variety of dimension $2$.   
   If $a_1, a_2, a_3$ are of modulus $1$ then so are $b,c$.  The quotient of $S^1 \times S^1$ by this equivalence relation is the Klein bottle
$\mathbf{Kb}$.   Therefore we have
\[
T^{\fs^0} / W^{\fs^0}  \quad \textrm{deformation-retracts onto} \quad \mathbf{Kb}.
\]

Let
\[
Y := (\C^\times)^2 / J,
\qquad
J =\{\pm 1\},\quad
(-1)\cdot(b,c)=(-b,-c).
\]
The map
\[
\tau \colon (b,c)\longmapsto (ic,\,ib)
\]
descends to a well-defined involution $\bar\tau$ of $Y$.
On the maximal compact real form, the quotient
\[
Y_{\mathrm{cpt}}/\langle \bar\tau\rangle
\]
is homeomorphic to a Klein bottle.
\end{theorem}

\begin{proof}
\textbf{$\tau$ becomes an involution on $Y$.}
On $(\C^\times)^2$ one computes
\[
\tau^2(b,c)=(-b,-c).
\]
Since this is precisely the $W$-action, $\tau^2$ is trivial on $Y$,
so $\bar\tau^2=\mathrm{id}$.

\medskip
\textbf{Angular coordinates.}
Restrict to the maximal compact torus
\[
T=S^1\times S^1,\qquad
b=e^{i\theta},\; c=e^{i\phi},\quad
(\theta,\phi)\in\R^2/(2\pi\Z)^2.
\]
The $J$-relation becomes
\[
(\theta,\phi)\sim(\theta+\pi,\phi+\pi),
\]
and the involution acts by
\[
\bar\tau(\theta,\phi)
= (\phi+\tfrac{\pi}{2},\,\theta+\tfrac{\pi}{2}).
\]

\medskip
\textbf{The action is free.}
If $\bar\tau(\theta,\phi)\equiv(\theta,\phi)$ modulo
$(2\pi\Z)^2$ and the $J$-relation, then
\[
\theta-\phi\equiv\tfrac{\pi}{2},\qquad
\phi-\theta\equiv\tfrac{\pi}{2}\pmod{2\pi},
\]
which is impossible. Hence the action is free, and the quotient
is a smooth compact surface without boundary.

\medskip
\textbf{Euler characteristic.}
Since $\bar\tau$ has order $2$ on $Y_{\mathrm{cpt}}$
and the original action on $T$ has order $4$ and is free,
\[
\chi\bigl(Y_{\mathrm{cpt}}/\langle\bar\tau\rangle\bigr)
= \chi(T)/4 = 0.
\]
Thus the quotient is either a torus or a Klein bottle.

\medskip
\textbf{Non-orientability.}
In angular coordinates the derivative of $\bar\tau$ is the matrix
\[
\begin{pmatrix}0&1\\[2pt]1&0\end{pmatrix},
\]
which has determinant $-1$; hence $\bar\tau$ is
orientation-reversing.
A free action containing an orientation-reversing element
produces a non-orientable closed surface.

\medskip
\textbf{Conclusion.}
A connected closed surface with Euler characteristic $0$
that is non-orientable is a Klein bottle.
\end{proof}

\section{Group cohomology} 
\subsubsection{The integers $\Z$}   
\begin{lemma}\label{G0}  We have $H^2(\Z, \C^\times) = 0$.
\end{lemma}

\begin{proof}  This follows from Corollary 6.2.7 in Weibel \cite{W} since $\Z$ is the free group on a point.  
 \end{proof}

\subsubsection{Cyclic group of order $2$}
Let $C_2=\Z/2\Z=\langle s\mid s^2=1\rangle$.

\begin{lemma}\label{G1}
$H^2(C_2,\mathbb{C}^\times)=0$.
\end{lemma}
\begin{proof}  For a cyclic group $C_n$ acting trivially on an abelian group $A$, one has \cite[Theorem 6.2.2]{W}
\[
H^2(C_n,A)\cong A/N(A),
\qquad N(a)=a\cdot (s\cdot a)\cdots(s^{n-1}\cdot a).
\]
For $n=2$ and $A=\mathbb{C}^\times$, the norm is $N(z)= z \cdot s(z) = z^2$, hence
\[
H^2(C_2,\mathbb{C}^\times)\cong \mathbb{C}^\times/(\mathbb{C}^\times)^2.
\]
Each coset is of the form $z(\C^\times)^2$ with $z \neq 0$.   But $z(\C^\times)^2 = \C^\times$ and so  the result follows.
\end{proof}

\subsubsection{The infinite dihedral group $D_\infty=\Z\rtimes \Z/2$}
%Let\[D_\infty=\Z\rtimes \Z/2 =\langle r,s\mid s^2=1,\ srs^{-1}=r^{-1}\rangle,\]where the generator $s$ of $\Z/2$ acts on $\Z=\langle r\rangle$ by inversion $r\mapsto r^{-1}$.

\begin{lemma}\label{G2}
$H^2(D_\infty,\mathbb{C}^\times)=0$.
\end{lemma}
\begin{proof}
Recall that the infinite dihedral group admits the presentation
\[
\mathrm{D}_\infty=\langle s,t \mid s^2=t^2=1\rangle \cong (\Z/2)\ast(\Z/2),
\]
\textit{i.e.}, it is the free product of two copies of $\Z/2$.
%Let $A$ be a trivial $D_\infty$--module. A standard topological model for the classifying space of a free product is a wedge:\[B(G*H)\simeq BG \,\vee\, BH.
%\]Since reduced cohomology of a wedge splits, for every $n\ge 2$ one has\[H^n(G*H,A)\cong H^n(G,A)\oplus H^n(H,A).\]
We now apply Corollary 6.2.10 in \cite{W}: 
%Applying this with $G=H=\Z/2$ and $A=\mathbb{C}^\times$ gives
\begin{eqnarray*}
H^2(\mathrm{D}_\infty,\mathbb{C}^\times) & \cong & H^2(\Z/2,\mathbb{C}^\times)\oplus H^2(\Z/2,\mathbb{C}^\times)\\
& = &0
\end{eqnarray*}
by Lemma \ref{G1}. 
\end{proof}

\subsubsection{The group $\Z/2\times \Z$}   Let $N= \Z/2 \times \Z$.    We consider group cohomology with coefficients in $\C^\times$ with trivial
$N$--action.  Write elements of $N$ as $(\varepsilon,m)$ with $\varepsilon\in\{0,1\}$ and $m\in\Z$.

Since $\C^\times$ is a divisible abelian group, the universal coefficient
theorem identifies
\[
H^2(N,\C^\times)\cong \mathrm{Hom}(H_2(D,\Z),\C^\times).
\]

Using the K\"unneth formula for group homology \cite[\S3.5]{W}, 
\[
H_2(\Z/2\times \Z,\Z)
\cong H_1(\Z/2,\Z)\otimes H_1(\Z,\Z)
\cong (\Z/2)\otimes \Z
\cong \Z/2,
\]
since $H_2(\Z/2,\Z)=H_2(\Z,\Z)=0$,
$H_1(\Z/2,\Z)=\Z/2$, and $H_1(\Z,\Z)=\Z$.

Therefore
\[
\boxed{
H^2(\Z/2\times \Z,\C^\times)\cong \Z/2.
}
\]
Hence there are exactly two cohomology classes: the trivial class and a unique
nontrivial class.

%We note that the explicit cocycle arising below from normalized intertwining operators differs from the abstract cocycle of \S3 by a coboundary. In particular both define the same class in $H^2(\Gamma,\C^\times)$ and hence the same twisted group algebra up to
%isomorphism.

\subsubsection{The preferred cocycle} Our goal is to describe the relevant class in
\[
H^2(N,\C^\times).
\]

A projective representation of $\Gamma$ determines a
$2$-cocycle only up to coboundary.
Different normalizations of generators therefore lead to
cohomologous cocycles.

%\subsubsection{Nontriviality of the cocycle $\mu$ on $\mathbb Z/2\times\mathbb Z$}

Let $N=\mathbb Z/2\times \mathbb Z$, written additively as pairs $(\varepsilon,m)$ with
$\varepsilon\in\{0,1\}$ and $m\in\mathbb Z$, and group law
\[
(\varepsilon,m)+(\delta,n)=(\varepsilon+\delta,\;m+n).
\]
We consider group cohomology with coefficients in $\mathbb C^\times$ and trivial $N$--action.  For compatibility with the normalization of intertwining operators
used later (Section~4), we adopt the cocycle
\[
\mu\bigl((\varepsilon,m),(\delta,n)\bigr)
\;=\;
(-1)^{\varepsilon n}.
\]

\begin{lemma}
The function $\mu$ is a normalized $2$--cocycle on $N$.
\end{lemma}

\begin{proof}
Normalization is immediate: if one entry is $(0,0)$ then the exponent is $0$ and $\mu=1$.

For the cocycle identity, take
$g=(\varepsilon,m)$, $h=(\delta,n)$, $k=(\gamma,p)$.
Then
\[
\mu(g,h)\mu(g+h,k)
=
(-1)^{\varepsilon n}\,(-1)^{(\varepsilon+\delta)p}
=
(-1)^{\varepsilon n+\varepsilon p+\delta p},
\]
while
\[
\mu(h,k)\mu(g,h+k)
=
(-1)^{\delta p}\,(-1)^{\varepsilon(n+p)}
=
(-1)^{\delta p+\varepsilon n+\varepsilon p}.
\]
These are equal, so $\mu$ is a $2$--cocycle.
\end{proof}

\begin{lemma}
The cohomology class $[\mu]\in H^2(N,\mathbb C^\times)$ is nontrivial.
\end{lemma}
\begin{proof}
For a $2$--cocycle $\alpha$ on an abelian group, define its commutator bicharacter
\[
b_\alpha(g,h):=\frac{\alpha(g,h)}{\alpha(h,g)}.
\]
If $\alpha=\partial f$ is a coboundary, then
\[
\partial f(g,h)=\frac{f(g)f(h)}{f(g+h)}=\partial f(h,g),
\]
so $b_{\partial f}(g,h)=1$ for all $g,h$.

Now compute $b_\mu$ on the generators $a=(1,0)$ of $\mathbb Z/2$ and $t=(0,1)$ of $\mathbb Z$:
\[
\mu(a,t)=(-1)^{1\cdot 1}=-1,\qquad
\mu(t,a)=(-1)^{0\cdot 0}=1.
\]
Hence
\[
b_\mu(a,t)=\frac{\mu(a,t)}{\mu(t,a)}=-1\neq 1.
\]
Therefore $\mu$ is not a coboundary, i.e.\ $[\mu]\neq 0$ in $H^2(G,\mathbb C^\times)$.
\end{proof}

%\subsubsection{Relation with other normalizations} Another natural cocycle on $G$ is\[\kappa\bigl((\varepsilon,m),(\delta,n)\bigr)=(-1)^{m\delta}.\]
%which arises from a different choice of section of the semidirect product. The cocycles $\kappa$ and $\mu$ are cohomologous: they differ by acoboundary coming from a $1$-cochain $b: G\to\{\pm1\}$.
%Consequently they define the same class in \[H^2(G,\C^\times),\]and hence yield isomorphic twisted group algebras.

In the remainder of the paper we use the cocycle $\mu$, since it matches
the normalization arising from intertwining operators in Section~4.

\subsubsection{The cocycle  $\mu^{\mathcal{T}}$ from \cite{AFO}}   The cocycle $\mu^{\mathcal{T}}$ has the following properties:
\begin{itemize}
\item it is trivial on $\Z_A$ and $\Z_B$ by Lemma   \ref{G0}
\item  it is trivial on the infinite dihedral group $\Z \rtimes \Z/2$ by Lemma \ref{G2}
\end{itemize}

Consider now the restriction of $\mu^{\mathcal{T}}$ to the subgroup 
\[
\Z/2 \times \Z \subset \Z \rtimes \Z/2 \times \Z.   
\]
It must be non-trivial by \cite[Corollary 3.6]{AFO}.   On the other hand, $\Z/2 \times \Z$ admits a unique nontrivial cohomology class.   Therefore we must have
\[
\mu^{\mathcal{T}}|_{\Z/2 \times \Z} = \mu,
\]
that is to say
\begin{align*}
\mu^{\mathcal{T}}((m_1,\varepsilon, n_1), (m_2, \delta, n_2)) &= \mu((\varepsilon, n_1), (\delta, n_2))\\
& = (-1)^{\varepsilon n_2}
\end{align*}
for all $m_1, m_2, n_1, n_2,  \in \Z$ and $\varepsilon, \delta  \in \Z/2$.   We shall use this formula in the next section.

\section{The twisted group algebra $\C[\Gamma, \mu^\cT]$ of Adler-Fintzen-Ohara}

We consider the countable discrete group
\[
\Gamma : = (\Z \rtimes \Z/2) \times \Z
\]
where the generator $s$ of $\Z/2$ acts as $s\cdot n = -n$.   We shall sometimes write
\[
\Gamma = (\Z_A \rtimes \Z/2) \times \Z_B
\]
in order to distinguish the two copies of $\Z$.   

The explicit cocycle $\mu = \mu^{\mathcal{T}}$ is given by
\[
\boxed{
\mu((m_1, \varepsilon, n_1), (m_2, \delta, n_2)) = (-1)^{\varepsilon \cdot n_2}
}
\]
with
 \[
 m_1, m_2 \in \Z_A \quad \quad \varepsilon, \delta \in \{0,1\}  \quad \quad n_1, n_2 \in \Z_B. 
 \] 

Note that 
\[
\mu |_{\Z_A \rtimes \Z/2} = 1,  \qquad \mu |_{\Z_B} = 1, \qquad \mu | _ {\Z/2} = 1.
\]

This is consistent with \S2.    The cocycle $\mu$ can be non-trivial  only when restricted to $\Z/2 \times \Z_B$.   
The twisted group algebra is $\C[\Gamma,\mu]$.   We shall write
\begin{align}\label{AB}
\C[\Z_A] &= \C[X, X^{-1}] \\
\C[\Z_B] &= \C[Y, Y^{-1}].
\end{align}

After the identification  (\ref{AB}) we have 
\begin{eqnarray}
s \cdot X = X^{-1}.   
\end{eqnarray}

The cocycle $\mu$, restricted to $\Z/2 \times \Z_B$,  is given by 
\[
\boxed{
\mu((\varepsilon, m), (\delta, n)): = (-1)^{\varepsilon n},
}
\] 
where $\varepsilon, \delta \in \{0,1\}$ and $m,n \in \Z$.     

We contemplate the pair $(s,Y)$.  We have 
\[
s = (1,0) \in \Z/2 \times_\mu \Z_B.
\]
   After the identification (\ref{AB}) we have 
   \[
   Y = (0,1) \in \Z/2 \times_\mu \Z_B.
   \]
   
   \begin{lemma}\label{sY} The algebra $\C[\Z/2 \times_\mu \Z]$ is a noncommutative $\C$-algebra with unit.   It admits generators $N_s, N_Y$ with the multiplication rule 
\[
N_s N_Y = \mu(s,Y)N_{sY}.
\]
The generators are $N_s, N_Y$ with relations
\[
N_s^2 = 1, \quad \quad N_sN_Y = - N_YN_s.
\]
   \end{lemma}
   
   \begin{proof}   The unit is $N_{(0,0)}$.   We have 
   \begin{eqnarray*}
   N_s^2 &=& N_{(1,0)}N_{(1,0)}\\
   & = &\mu((1,0),(1,0)))N_{(0,0)}\\
   & = &1.
   \end{eqnarray*}
   
   We also have
   \begin{eqnarray*}
   N_sN_Y &=& \mu(s,Y)N_{sY}\\
   &=& - N_{sY}\\
   &=& - N_{Ys}
   \end{eqnarray*}
   and
   \begin{eqnarray*}
   N_YN_s &=& \mu(Y,s)N_{Ys}\\
   &= & N_{Ys}
   \end{eqnarray*}
   as required.   
   \end{proof}
   
   \begin{corollary}  Working in the group algebra $\C[\Z/2 \times_\mu \Z_B]$ we have, by Lemma \ref{sY}, 
\begin{eqnarray*}
N_sN_YN_s^{-1} &= &- N_YN_sN_s^{-1}\\
&=& - N_Y 
\end{eqnarray*}
so that 
\begin{eqnarray}
s \cdot Y = - Y
\end{eqnarray}
precisely as in Goldberg-Roche \cite[\S11.8]{GR2}.
\end{corollary} 

%Following \cite[\S3]{AFO} we define the element $\widetilde{s}$ of the group \[G^0(F)  = (\GL_2(E_2) \times \GL_1(E_4)) \cap \SL_8(F) \supset \SL_2(E_2) \times \SL_1(E_4)\]by \[
%\widetilde{s} = ((\begin{smallmatrix}0 & 1\\-1 & 0\end{smallmatrix}),1).\]Conjugation by $\widetilde{s}$ will permute the entries in the diagonal subgroup of $\GL_2(E_2)$ and so will be the non-trivial element in the Weyl group $W(G^0, M^0)$.   
%In the positive-depth case \cite{GR} and the depth-zero case \cite{AFO} we have $T^\fs$ is a complex torus of dimension two.

\section{Simple modules of $\C[\Gamma, \mu^\cT]$}   
\begin{theorem}\label{thm:simples-expanded}
We have, for the twisted AFO-algebra,
\[
\C[\Z\rtimes\Z/2 \times \Z, \mu]
\;\cong\;
\C[s,X,X^{-1},Y,Y^{-1}]
\Big/\bigl(s^2-1,\; sY+Ys,\; sX-X^{-1}s,\; XY-YX\bigr),
\]
and the simple modules are the $2$--dimensional modules
\[
M_{w,z}\cong \C^2,\qquad
Y\mapsto\begin{pmatrix} w&0\\0&-w\end{pmatrix},\quad
s\mapsto\begin{pmatrix}0&1\\1&0\end{pmatrix},\quad
X\mapsto\begin{pmatrix} z&0\\0&z^{-1}\end{pmatrix},
\]
with parameters $(w,z)\in(\C^\times)^2$, satisfying
\[
M_{w,z}\;\cong\; M_{-w,z^{-1}},
\]
and every simple module is isomorphic to one of these.
%Consequently the parameter space for simple modules is \[(\C^\times\times\C^\times)\big/\bigl((w,z)\sim(-w,z^{-1})\bigr).\]
The primitive spectrum of $\C[\Gamma, \mu^{\mathcal{T}}]$  is naturally identified with $(\C^\times \times \C^\times) / (w,z) \sim (-w, z^{-1})$.   
In particular, the primitive spectrum is irreducible.
\end{theorem}

\begin{proof}
\emph{Presentation of the twisted group algebra.}
Write
\[
\Gamma=(\Z_A\rtimes \Z/2)\times \Z_B,
\]
where the generator $s$ of $\Z/2$ acts on $\Z_A$ by inversion.
Let $X$ denote the group-like element corresponding to $1\in \Z_A$
and $Y$ the group-like element corresponding to $1\in \Z_B$.
Then in the untwisted group algebra one has
\[
sXs^{-1}=X^{-1},\qquad XY=YX,\qquad s^2=1.
\]
The cocycle $\mu$ is trivial on $\Z_A\rtimes \Z/2$ and on $\Z_B$,
and is nontrivial only on the subgroup $\Z/2\times \Z_B$; explicitly,
with additive notation $(\varepsilon,m)\in \Z/2\times\Z_B$,
\[
\mu\bigl((\varepsilon,m),(\delta,n)\bigr)=(-1)^{\varepsilon n}.
\]
In particular, if we write $N_g$ for the basis element attached to $g\in\Gamma$,
then in $\C[\Gamma,\mu]$ we have
\[
N_sN_Y=\mu(s,Y)N_{sY}=(-1)N_{sY},\qquad
N_YN_s=\mu(Y,s)N_{Ys}=(+1)N_{Ys}.
\]
Since $Ys=sY$ in the underlying group, this yields the twisted relation
\[
N_sN_Y=-\,N_YN_s.
\]
Thus $\C[\Gamma,\mu]$ is generated by $s,X^{\pm1},Y^{\pm1}$ with relations
\[
s^2=1,\qquad sX=X^{-1}s,\qquad sY=-Ys,\qquad XY=YX,
\]
which gives the stated presentation.

\medskip
\emph{No one-dimensional modules.}
Let $M$ be a $1$--dimensional module. Since $Y$ is invertible in the algebra,
its image in $\mathrm{End}(M)\cong \C$ must be a nonzero scalar, say $Y\mapsto \lambda\in\C^\times$.
But the relation $sY=-Ys$ forces
\[
s\lambda = -\lambda s.
\]
In $\C$ this implies $2\lambda s=0$, hence $s=0$, contradicting $s^2=1$.
Therefore \emph{no} $1$--dimensional representations exist.

\medskip
\emph{Reduction to Clifford theory over a commutative subalgebra.}
Let
\[
B:=\C[X^{\pm1},Y^{\pm1}]\subset \C[\Gamma,\mu].
\]
Then $B$ is a commutative Laurent polynomial algebra in two variables.
Conjugation by $s$ defines an involutive automorphism $\alpha$ of $B$:
\[
\alpha(X)=X^{-1},\qquad \alpha(Y)=-Y.
\]
Equivalently, for any character $\chi:B\to\C$ determined by
\[
\chi(X)=z\in\C^\times,\qquad \chi(Y)=w\in\C^\times,
\]
one has
\[
(\chi\circ\alpha)(X)=z^{-1},\qquad (\chi\circ\alpha)(Y)=-w.
\]
Thus the $\Z/2$--action on $\Spec(B)\cong(\C^\times)^2$ is
\[
(w,z)\longmapsto(-w,z^{-1}).
\]
Note that this action has \emph{no fixed points} on $(\C^\times)^2$:
a fixed point would require $z=z^{-1}$ and $w=-w$, hence $w=0$, impossible.

\medskip
\emph{Construction of the $2$--dimensional simple modules.}
Fix $(w,z)\in(\C^\times)^2$ and let $\C_{w,z}$ be the $1$--dimensional $B$--module
with $X\mapsto z$ and $Y\mapsto w$.
Form the induced module
\[
M_{w,z}:=\Ind_{B}^{\C[\Gamma,\mu]}\C_{w,z}\;\cong\;\C[\Gamma,\mu]\otimes_B \C_{w,z}.
\]
Because the $\Z/2$--orbit of $(w,z)$ has size $2$, standard Clifford theory
(or a direct computation below) shows $M_{w,z}$ is $2$--dimensional and simple.

Concretely, let $v:=1\otimes 1\in M_{w,z}$ and set $v':=s\cdot v$.
Then $\{v,v'\}$ is a basis of $M_{w,z}$.
Using $s^2=1$ we have $s\cdot v=v'$ and $s\cdot v'=v$, so
\[
s\mapsto \begin{pmatrix}0&1\\[2pt]1&0\end{pmatrix}.
\]
Next, $Xv=zv$ by definition.
Also,
\[
Xv' = X(sv) = (Xs)v = (sX^{-1})v = s(z^{-1}v)=z^{-1}v',
\]
so
\[
X\mapsto \begin{pmatrix} z&0\\[2pt]0&z^{-1}\end{pmatrix}.
\]
Similarly, $Yv=wv$ and
\[
Yv' = Y(sv) = -(sY)v = -s(wv)= -wv',
\]
hence
\[
Y\mapsto \begin{pmatrix} w&0\\[2pt]0&-w\end{pmatrix}.
\]
These matrices satisfy the defining relations, so $M_{w,z}$ is indeed a representation.

\medskip
\emph{Isomorphism relation $M_{w,z}\cong M_{-w,z^{-1}}$.}
The $\Z/2$--action on characters of $B$ sends $(w,z)$ to $(-w,z^{-1})$.
Equivalently, conjugating the above matrices by the $s$--matrix swaps the two basis vectors
and sends $(w,z)$ to $(-w,z^{-1})$.
Thus $M_{w,z}\cong M_{-w,z^{-1}}$.

\medskip
\emph{Exhaustion: every simple module is one of these.}
Let $M$ be a simple $\C[\Gamma,\mu]$--module.
Restrict $M$ to the commutative subalgebra $B$.
Choose a maximal ideal $\frak m\subset B$ in the support of $M$; equivalently,
choose a character $\chi:B\to\C$ occurring in $M$, with $\chi(X)=z\neq 0$ and $\chi(Y)=w\neq 0$
(since $X,Y$ are invertible in the algebra).
Because the $\Z/2$--action on $\Spec(B)$ is free, the $\chi$--isotypic component generates a
$\C[\Gamma,\mu]$--submodule of dimension $2$ (spanned by $v$ and $sv$ as above), and simplicity
forces $M$ to be isomorphic to this induced module $M_{w,z}$.
Therefore every simple module is $2$--dimensional and isomorphic to some $M_{w,z}$, with the
sole identification $(w,z)\sim(-w,z^{-1})$.

This completes the proof.
\end{proof}

\begin{corollary}[Klein bottle as the compact parameter space]\label{cor:Klein-compact}
Let
\[
\tau:\; S^1\times S^1 \longrightarrow S^1\times S^1,\qquad 
\tau(w,z)=(-w,z^{-1}).
\]
Then the parameter space of simple \emph{unitary} modules of the twisted algebra
$\C[\Gamma,\mu^\cT]$ (i.e.\ those with $|w|=|z|=1$) is the quotient
\[
K := (S^1\times S^1)/\langle\tau\rangle,
\]
and $K$ is homeomorphic to a Klein bottle.
\end{corollary}

\begin{proof}
By Theorem~\ref{thm:simples-expanded}, simple modules are parametrized by
$(w,z)\in(\C^\times)^2$ modulo the identification $(w,z)\sim(-w,z^{-1})$.
Restricting to the unitary locus $|w|=|z|=1$ gives precisely the quotient
$K=(S^1\times S^1)/\langle\tau\rangle$.

Write $w=e^{i\theta}$ and $z=e^{i\phi}$. Then $\tau$ acts by
\[
(\theta,\phi)\longmapsto(\theta+\pi,\,-\phi).
\]
This action is free: a fixed point would require $\theta\equiv \theta+\pi \pmod{2\pi}$,
impossible. Hence $S^1\times S^1\to K$ is a $2$--fold covering and
\[
\chi(K)=\chi(S^1\times S^1)/2 = 0.
\]
Moreover $\tau$ is orientation-reversing (it flips the $\phi$--circle), so $K$ is a closed
non-orientable surface with Euler characteristic $0$. Therefore $K$ is a Klein bottle.
\end{proof}

\section{Simple modules of the untwisted algebra $\C[\Gamma]$}

\begin{proposition}\label{prop:simples-untwisted}
Let
\[
\Gamma := (\Z_A \rtimes \Z/2)\times \Z_B,
\]
where the generator $s$ of $\Z/2$ acts on $\Z_A$ by inversion.
Let $X$ (resp. $Y$) denote the group-like element corresponding to $1\in \Z_A$
(resp.\ $1\in \Z_B$). Then the (untwisted) group algebra admits the presentation
\[
\C[\Gamma]\;\cong\;
\C[s,X,X^{-1},Y,Y^{-1}]\Big/\bigl(s^2-1,\; sX-X^{-1}s,\; sY-Ys,\; XY-YX\bigr),
\]
i.e.\ $B:=\C[X^{\pm1},Y^{\pm1}]$ is commutative and $s$ acts on $B$ via
\[
\alpha(X)=X^{-1},\qquad \alpha(Y)=Y.
\]
The simple $\C[\Gamma]$--modules are as follows.

\smallskip\noindent
%\emph{Two-dimensional simple modules (generic orbit).}
For $(w,z)\in (\C^\times)^2$ with $z\neq \pm 1$, let $M_{w,z}\cong \C^2$ with
\[
Y\mapsto\begin{pmatrix} w&0\\0&w\end{pmatrix},\qquad
s\mapsto\begin{pmatrix}0&1\\1&0\end{pmatrix},\qquad
X\mapsto\begin{pmatrix} z&0\\0&z^{-1}\end{pmatrix}.
\]
Then $M_{w,z}$ is simple, and
\[
M_{w,z}\;\cong\; M_{w,z^{-1}}.
\]
Every simple module of dimension $2$ is isomorphic to exactly one $M_{w,z}$ with $z\neq \pm1$,
up to the identification $z\sim z^{-1}$.

\smallskip\noindent
%\textbf{(II) One-dimensional simple modules (fixed points $z=\pm 1$).}
For each $w\in \C^\times$, each choice of sign $\delta\in\{+1,-1\}$, and each
$\varepsilon\in\{+1,-1\}$, let $\C_{w,\delta,\varepsilon}$ be the $1$--dimensional module with
\[
Y\mapsto w,\qquad X\mapsto \delta,\qquad s\mapsto \varepsilon.
\]
These are simple, pairwise non-isomorphic, and \emph{exhaust} all $1$--dimensional simple modules.

\smallskip
Consequently, the set of isomorphism classes of simple modules of $\C[\Gamma]$ is the disjoint union
\[
\Bigl\{\,M_{w,z}:\ (w,z)\in(\C^\times)^2,\ z\neq\pm 1\,\Bigr\}\Big/\bigl((w,z)\sim(w,z^{-1})\bigr)
\;\;\sqcup\;\;
\Bigl\{\,\C_{w,\delta,\varepsilon}:\ w\in\C^\times,\ \delta=\pm1,\ \varepsilon=\pm1\,\Bigr\}.
\]
\end{proposition}

\begin{proof}
\emph{Presentation.}
In the group $\Gamma$ we have $s^2=1$, $sXs^{-1}=X^{-1}$ (inversion on $\Z_A$),
and $sYs^{-1}=Y$ (since $\Z_B$ is a direct factor). Also $XY=YX$ because $\Z_A$ and $\Z_B$
commute. This gives the stated presentation.

\medskip
\emph{View as a crossed product.}
Let $B=\C[X^{\pm1},Y^{\pm1}]$. Conjugation by $s$ induces an involution $\alpha$ of $B$:
$\alpha(X)=X^{-1}$ and $\alpha(Y)=Y$. Thus $\C[\Gamma]\cong B\rtimes_\alpha \Z/2$.

Characters of $B$ are parametrized by $(w,z)\in(\C^\times)^2$ via
\[
\chi_{w,z}(Y)=w,\qquad \chi_{w,z}(X)=z.
\]
The induced action on $\Spec(B)\cong(\C^\times)^2$ is
\[
(w,z)\longmapsto (w,z^{-1}).
\]
Hence the orbit of $(w,z)$ has size $2$ if and only if $z\neq \pm 1$, and has size $1$ exactly
when $z=\pm1$.

\medskip
\emph{Generic case $z\neq \pm1$: induced modules are $2$--dimensional and simple.}
Fix $z\neq \pm1$ and consider the $1$--dimensional $B$--module $\C_{w,z}$ with character $\chi_{w,z}$.
Form the induced module
\[
M_{w,z}:=\Ind_{B}^{B\rtimes\Z/2}\C_{w,z}\;\cong\; (B\rtimes\Z/2)\otimes_B \C_{w,z}.
\]
Let $v:=1\otimes 1$ and $v':=s\cdot v$. Then $\{v,v'\}$ spans $M_{w,z}$.
One computes
\[
Yv=wv,\quad Yv'=w v',\qquad
Xv=zv,\quad Xv'=X(sv)=(sX^{-1})v=z^{-1}v',
\]
and $s$ swaps $v$ and $v'$, giving exactly the displayed matrices.
Since the $\Z/2$--orbit of $\chi_{w,z}$ has size $2$, standard Clifford theory for crossed products
implies $M_{w,z}$ is simple; moreover $M_{w,z}\cong M_{w,z^{-1}}$ because $(w,z)$ and $(w,z^{-1})$
lie in the same orbit.

\medskip
\emph{Fixed-point case $z=\pm1$: one-dimensional simple modules.}
Assume $z=\delta\in\{\pm1\}$ so $\chi_{w,\delta}$ is fixed by $\alpha$.
Then the relation $sX=X^{-1}s$ imposes no constraint beyond $s\delta=\delta s$, and $s$ commutes
with $Y$ in the untwisted algebra. Hence we may choose any scalar $\varepsilon=\pm1$ for $s$
(subject only to $s^2=1$), obtaining the $1$--dimensional module $\C_{w,\delta,\varepsilon}$.
These are clearly pairwise non-isomorphic by their eigenvalues for $(X,Y,s)$.

Finally, if $z=\pm1$ then the induced module $\Ind_B^{B\rtimes\Z/2}\C_{w,z}$ is reducible and splits
as the direct sum of the two characters corresponding to $\varepsilon=\pm1$.
Thus there are no further simple modules.
\end{proof}

\begin{remark}[Effect of the cocycle on simple modules and topology]\label{rem:twist-comparison}
The passage from the untwisted group algebra $\C[\Gamma]$ to the twisted algebra
$\C[\Gamma,\mu^T]$ has two decisive consequences.

First, the cocycle modifies the crossed--product relation between the $\Z/2$--generator $s$
and the $\Z_B$--generator $Y$ from $sY=Ys$ to $sY=-Ys$.  
As a result, $1$--dimensional representations disappear: in dimension $1$ the relation
$sY=-Ys$ would force $Y=0$, contradicting invertibility.  
Thus every simple module of the twisted algebra becomes $2$--dimensional
(Theorem~\ref{thm:simples-expanded}).

Second, this altered commutation relation changes the induced action of $\Z/2$
on the character torus $\Spec \C[X^{\pm1},Y^{\pm1}] \cong (\C^\times)^2$.
In the untwisted case the involution is
\[
(w,z)\longmapsto (w, z^{-1}),
\]
whose fixed points $z=\pm1$ give rise precisely to the $1$--dimensional simple modules.
After twisting, the involution becomes
\[
(w,z)\longmapsto (-w, z^{-1}),
\]
which has no fixed points on $(\C^\times)^2$ and therefore yields only
$2$--dimensional simple modules.

On the compact real form $S^1\times S^1$, this change replaces the quotient
of the torus by a reflection with boundary (compatible with the presence of
$1$--dimensional representations) by a free, orientation--reversing involution.
The resulting compact quotient is non--orientable and closed, hence a Klein bottle.
In this precise sense, the Klein bottle geometry is the topological manifestation
of the nontrivial cocycle $\mu^\cT$.
\end{remark}

\section{The Klein bottle as geometric footprint of the algebra $\C[\Gamma, \mu^\cT]$}   

Consider the involution
\[
\tau\colon\; S^1\times S^1 \longrightarrow S^1\times S^1,\qquad \tau(w,z)=(-w,z^{-1}),
\]
and the quotient space
\[
\cK := (S^1\times S^1)/\langle \tau\rangle .
\]
Write $w=e^{i\theta}$ and $z=e^{i\phi}$.  Then $\tau$ acts by
\[
(\theta,\phi)\longmapsto (\theta+\pi,\,-\phi).
\]
This action is free (no fixed points), hence $\cK$ is a closed smooth surface and
$S^1\times S^1\to \cK$ is a $2$--fold covering.

\subsubsection{How it sits inside the complex quotient as a real form}
Let
\[
X := (\C^\times\times \C^\times)/\langle \tau\rangle,\qquad \tau(w,z)=(-w,z^{-1}).
\]
The compact torus $S^1\times S^1\subset \C^\times\times\C^\times$ is $\tau$--stable, hence its
quotient $\cK$ embeds naturally as a real submanifold of $X$:
\[
K \;=\; (S^1\times S^1)/\langle\tau\rangle \;\hookrightarrow\;
(\C^\times\times\C^\times)/\langle\tau\rangle \;=\; X.
\]
Moreover $\cK$ is a \emph{maximal compact} real form of $X$ in the sense that
$\C^\times\times\C^\times$ deformation-retracts onto $S^1\times S^1$, and this retract is
$\tau$--equivariant; hence $X$ deformation-retracts onto $\cK$.  In particular, $\cK$ captures the
homotopy type of the complex surface $X$.

Restricting to the maximal compact subgroup $S^1\times S^1$ produces the compact real form of
the quotient.  Because the involution is free but orientation-reversing on the torus,
the resulting compact quotient is non-orientable, namely the Klein bottle.  In this sense the
Klein bottle is a topological footprint of the nontrivial cocycle $\mu^\cT$.

\begin{lemma}[Square--root coordinates for the Bernstein torus]\label{lem:bc-wz}
Let
\[
J=\{\pm 1\}, \qquad (-1)\cdot(b,c)=(-b,-c)
\]
act on $(\C^\times)^2$, and define
\[
w := bc, \qquad z := \frac{b}{c}.
\]
Then the map
\[
(\C^\times)^2 \longrightarrow (\C^\times)^2,
\qquad (b,c)\longmapsto (w,z)
\]
is invariant under $J$ and induces an algebraic isomorphism
\[
(\C^\times)^2/J \;\xrightarrow{\;\sim\;}\; (\C^\times)^2.
\]
\end{lemma}

\begin{proof}
If $(b,c)\mapsto(-b,-c)$ then
\[
bc \longmapsto (-b)(-c)=bc,\qquad 
\frac{b}{c} \longmapsto \frac{-b}{-c}=\frac{b}{c},
\]
so the map is $J$--invariant and hence descends to the quotient.

Conversely, given $(w,z)\in(\C^\times)^2$, any solution $(b,c)$ of
\[
b^2 = wz, \qquad c^2 = \frac{w}{z}
\]
determines a preimage, and the only ambiguity is the simultaneous sign
$(b,c)\sim(-b,-c)$, which is precisely the $J$--action.
Hence the induced map $(\C^\times)^2/J \to (\C^\times)^2$ is bijective.
Since both varieties are smooth algebraic tori of dimension $2$, the map
is an algebraic isomorphism.
\end{proof}

\subsubsection{The tempered dual of $G^0$}   We recall the definitions of $w$ and $z$:
\[
w: = bc, \qquad \quad z: = \frac{b}{c}.
\]
It follows that $w$ and $z$ have modulus $1$ if $b$ and $c$ have modulus $1$.   

Let $\fB$ denote the reduced $C^\ast$-algebra of $G^0$.   We have the Bernstein decomposition
\[
\fB  = \bigoplus_{\ft\in\fB G^0} \, \fB^{\ft}
\]
as a $C^\ast$-direct sum.   Define
\[
\fA^0: = \fB^{\fs^0}
\]
and let $X_{\unr}(M^0)$ denote the group of unramified unitary characters of $M^0$.

 Let $\widehat{\fA^0}$ denote the spectrum of $\fA^0$, equipped with its standard Jacobson topology on the set of its primitive ideals, see \cite[\S3.1]{D}.  
 
 \begin{theorem}\label{spectrum} The spectrum of $\fA^0$ is homeomorphic to the Klein bottle:
 \[
 \widehat{\fA^0} \cong \mathbf{Kb}.
 \]
 \end{theorem}
 \begin{proof} We note that $M^0$ is a maximal Levi subgroup of $G^0$ and that $W^{\fs^0}$ acts freely on $X_{\unr}(M^0)$.   We can now apply the Proposition in \cite[\S4.3]{R} 
to infer that each induced representation $\mathrm{Ind}_{M^0N^0}^{G^0} \, \psi \otimes \sigma^0$ with $ \psi \in X_{\unr}(M^0)$ is an irreducible unitary representation of $G^0$,  a \emph{tempered}
representation of $G^0$.      Therefore, the map
\begin{align}\label{Xunr}
X_{\unr}(M^0) & \longrightarrow \widehat{\fA^0} \\
\psi & \mapsto \mathrm{Ind}_{M^0N^0}^{G^0} \, \psi \otimes \sigma^0  \nonumber
\end{align}
is well-defined and surjective.   The map (\ref{Xunr}) determines a bijective map
\[
X_{\unr}(M^0) / W^{\fs^0} \longrightarrow \widehat{\fA^0}
\]
that is to say,  a bijective map
\[
\mathbf{Kb} \longrightarrow \widehat{\fA^0}.
\]
%By transport of structure, $\widehat{\fA}$ admits the same   Klein bottle structure.   The spectrum of $\fB$ admits its standard topology coming from the Zariski topology of primitive ideals in $\fB$.   

At this point in the proof, we shall use the $C^*$-Plancherel Theorem from \cite{P}.   According to Corollary 2.7 in \cite{P}, the Fourier Transform $\mathcal{F}$ determines an isomorphism of $C^*$-algebras
\[
\mathcal{F} : \fA^0 \cong C(\mathbf{Kb}, \mathfrak{K})
\]
where $\mathbf{Kb}$ is the Klein bottle in its natural topology (as quotient of $S^1  \times S^1$), $\mathfrak{K}$ is the $C^*$-algebra of compact operators on separable infinite-dimensional Hilbert space, and $C(\mathbf{Kb}, \fK)$ 
denotes the $C^*$-algebra of all continuous maps from  $\mathbf{Kb}$ to $\fK$.  It follows that $\fA^0$ is strongly Morita equivalent to $C(\mathbf{Kb})$.   Therefore, we have
\[
\widehat{\fA^0} \cong \mathbf{Kb}.
\]
This statement says that the spectrum of $\fA^0$  is homeomorphic to $\mathbf{Kb}$.  
\end{proof} 

The Klein bottle $\mathbf{Kb}$ therefore appears as a  connected compact Hausdorff subspace of the tempered dual of $G^0$.

\section{Positive depth}

We briefly indicate how the preceding analysis may extend to the
positive-depth setting considered by Roche in \cite[\S4.1]{R}.

Let $G:=\SL_8$. With the same basic data $(F,E_2/F,E_4/F,\zeta,\eta,\chi_0)$,
Roche considers  a Levi subgroup $M(F) \subset \SL_8(F)$ such that 
\[M(F)\simeq (\GL_2(F)\times \GL_2(F)\times \GL_4(F))\cap \SL_8(F)\]
and an irreducible supercuspidal representation $\sigma$ of $M(F)$.

For a representation $\widetilde \sigma_i$ of $\GL_{n_i}(F)$, we denote by $\mathfrak{S}(\widetilde \sigma_i)$ the set of characters  $\chi$ of $F^\times$ such that $\widetilde \sigma_i\simeq\chi\widetilde \sigma_i$, where $\chi\widetilde \sigma_i$ 
is the representation $g\mapsto\widetilde \sigma_i(g)\chi(\det(g))$ for $g\in\GL_{n_i}(F)$.   The representations $\widetilde{\sigma_1}$ and $\widetilde{\sigma_2}$, where
$\widetilde\sigma_1$ is an irreducible  supercuspidal representation of $\GL_2(F)$ and  $\widetilde\sigma_2$ an irreducible supercuspidal representation of $\GL_4(F)$, are constructed so that
 \[
 \mathfrak{S}(\widetilde\sigma_1) = \{1,\eta^2\}\quad \quad \textrm{and} \quad \quad \mathfrak{S}(\widetilde\sigma_2) = \langle \chi_0\eta\rangle.
 \]

The representation $\sigma$ is then defined to be the restriction to $M(F)$ of the irreducible supercuspidal representation 
\[\widetilde\sigma:=\widetilde\sigma_1 \otimes \eta\widetilde\sigma_1 \otimes \widetilde\sigma_2\] 
of $\GL_2(F) \times \GL_2(F) \times \GL_4(F)$ where $\sigma_1, \sigma_2$ are the supercuspidals carefully constructed  in \cite[\S4.1]{R}.

The  cuspidal pair $(M,\sigma)$ gives rise to a point $\fs=[M(F),\sigma]_G$ in the Bernstein spectrum $\fB\SL_8$.  
Let
\[
\cH^\fs := \cH(\SL_8(F),(K_\fs,\rho_\fs))
\]
denote the Hecke algebra attached to a type $(K_\fs,\rho_\fs)$
for $\fs$.   The existence of $(K_\fs,\rho_\fs)$ follows from \cite[Theorem~4.4]{GR1}.

By \cite[\S4.2]{R}, the group $W^\fs$ is of order $2$.  Let $\epsilon$ denote the nontrivial element in $W^\fs$.   

Let $\delta$ be the $2$-cocycle which occurs in \cite[Theorem 11.1]{GR2}.  The group $C$ which occurs in their Theorem  11.1 is $\Z/2\Z$ in our case, for which $\delta$ is automatically a coboundary. In that case, we have
$\C[\Z/2\Z]_\delta = \C[\Z/2\Z]$, see \cite[\S11.1, \S11.7]{GR2}.

\subsubsection{A description of the Hecke algebra $\cH^\fs$}

\begin{theorem}\label{cHcH} We have an isomorphism of $\C$-algebras: 
\[\cH^\fs\simeq \C[\Gamma,\mu].\]
%In particular, the $2$-cocycle $\delta$  from \cite{GR2} is trivial, while $\mu$ is nontrivial.
\end{theorem}
\begin{proof}
By \cite[\S11.8]{GR2}, we have an isomorphism 
\[
\cH^{\fs} \cong \C[X^{\pm 1},Y^{\pm 1}]\widetilde\otimes_t \C[\Z/2\Z]
\]
in the notation of  the introduction of \cite{GR2}. Here 
\[t\colon \Z/2\Z\to\mathrm{Aut}(\C[X^{\pm 1},Y^{\pm 1}])\]
is the injective homomorphisms of groups defined by 
\begin{align}\label{t-e}
t(\epsilon)(X):=X^{-1}\quad\text{and}\quad t(\epsilon)(Y):=-Y.
\end{align}

We recall the isomorphism
\begin{align*}
\C[\Z\rtimes\Z/2 \times \Z, \mu]  & \cong \C[s,X,X^{-1},Y,Y^{-1}] \Big/\bigl(s^2-1,\; sY+Ys,\; sX-X^{-1}s,\; XY-YX\bigr).
 %\cH^\fs  & \cong [\C[X^{\pm 1},Y^{\pm 1}]\widetilde\otimes_t \C[\Z/2\Z,\delta].
 \end{align*}
The map $s \mapsto t(\varepsilon)$ effects an isomorphism: 
\begin{align*}
 \C[s,X,X^{-1},Y,Y^{-1}] \Big/\bigl(s^2-1,\; sY+Ys,\; sX-X^{-1}s,\; XY-YX\bigr)  & \cong [\C[X^{\pm 1},Y^{\pm 1}]\widetilde\otimes_t \C[\Z/2\Z]\\
 s & \mapsto t(\varepsilon)
 \end{align*}
 in view of the equations (\ref{t-e}).   
\end{proof}

\subsubsection{Comparison}   It follows from Theorem \ref{cHcH} that $\cH^\fs$ and $\cH^{\fs^0}$ will share the same simple modules.     We have 
\[
T^\fs \cong T^{\fs^0} \cong (\C^\times)^2, \qquad W^\fs \cong W^{\fs^0} \cong \Z/2\Z.
\]

Much more than this is true.

\begin{theorem}  Let $G = \SL_8(F)$.  Let $\fs$ be the point in the Bernstein spectrum $\fB G$ determined by the cuspidal pair $(M,\sigma)$.   
  Let $T^{\fs}$ be the Bernstein torus attached to  $\fs$.   Then we have a $\Z/2\Z$-equivariant isomorphism
\[
 \tau :  T^{\fs} \cong T^{\fs^0}.   
\]
\end{theorem}
\begin{proof}   The map $\tau$ is defined as follows:
\begin{align*} 
\tau : X_{\nr}(M) & \longrightarrow X_{\nr}(M^0) \\
a^{\val \circ \det}\;b^{\val \circ \det}\;c^{\val \circ \det} & \mapsto a^{\val \circ N_{E_2/F}}\; b^{\val \circ N_{E_2/F}}\; c^{\val \circ N_{E_4/F}}
\end{align*}
with $a,b,c  \in \C^\times$.

Given 
\[
(A, B, C) \in (\GL_2(F) \times  \GL_4(F) \times \GL_4(F)) \cap \SL_8(F), 
\]
we note that $\det(A) \det(B) \det(C) = 1$.   Therefore we have
\[
\val(\det(A)) + \val(\det(B)) + \val(\det(C)) = 0.
\]

So the unramified character in the domain of $\tau$ is trivial when $a = b = c$.   

Following \cite[\S4.1]{R} we let $\chi_0$ denote the unramified character of $F^\times$ such that $\chi_0(\varpi_F) = - \sqrt{-1}$.   Then, for all $C \in \GL_4(F)$, we have
\[
(\chi_0 \sigma) (C) = -i ^{\val(\det(C))} \sigma(C)
\]
According to \cite[\S4.2]{R} we have
\[
{}^w \sigma \cong (1 \otimes 1 \otimes \chi_0) \sigma.
\]
Applying $w$ one more time, we obtain
\begin{align*}
\sigma & \cong (1 \otimes 1 \otimes \chi_0^2) \sigma   \\
& \cong (1 \otimes 1 \otimes (-1)^{\val \circ \det}\,)\sigma.
\end{align*}

In both the domain and co-domain of the map $\tau$, the characters are parametrized by triples $(a,b,c)$.   In both the domain and co-domain of the map $\tau$, the character is trivial if and only if 
the triple is $(a,a,a)$ or $(b,b-b)$.   Therefore, the map $\tau$ is an isomorphism of abelian groups.   

The character $1 \otimes 1 \otimes \chi_0$ has homogeneous coordinates $(1 : 1 : -i)$ and the character $1 \otimes 1 \otimes \chi_0^2$ has homogeneous coordinates $(1 : 1 : -1)$.   

The map $\tau$ is $\Z/2\Z$-equivariant for the following reason.   In both the domain and co-domain of $\tau$, the non-trivial element $s$ of $\Z/2\Z$ acts on homogeneous triples as follows:
\[
s \cdot (a:b:c) =  (b:a: -ic) = (ib: ic : c).
\]
\end{proof}

\begin{theorem}   There is a canonical isomorphism of Bernstein varieties:
\[
T^\fs / W^\fs \cong T^{\fs^0} / W^{\fs^0}.
\]
The Bernstein variety $T^\fs / W^\fs$ is a complex algebraic variety of dimension $2$ whose maximal compact real form is a Klein bottle $\mathbf{Kb}$.   This Klein bottle is a  compact Hausdorff subspace of the  tempered dual of 
$\SL_8$:   
\[
\mathbf{Kb} \subset \Irr^\temp(\SL_8)
\]
\end{theorem}

\begin{proof}
Let $\fC$ denote the reduced $C^\ast$-algebra of $G$.   We have the Bernstein decomposition
\[
\fC  = \bigoplus_{\ft\in\fB G} \, \fC^{\ft}
\]
as a $C^\ast$-direct sum.   Define
\[
\fA: = \fC^{\fs}
\]
and let $X_{\unr}(M)$ denote the group of unramified unitary characters of $M$.

 Let $\widehat{\fA}$ denote the spectrum of $\fA$, equipped with its standard Jacobson topology on the set of its primitive ideals, see \cite[\S3.1]{D}.  
 
According to \cite[\S4.1]{R}, each induced representation $\mathrm{Ind}_{MN}^{G} \, \psi \otimes \sigma$ with $ \psi \in X_{\unr}(M^)$ is an irreducible unitary representation of $G$,  a \emph{tempered}
representation of $G$.      Therefore, the map
\begin{align}\label{Xunr}
X_{\unr}(M) & \longrightarrow \widehat{\fA} \\
\psi & \mapsto \mathrm{Ind}_{MN}^{G} \, \psi \otimes \sigma  \nonumber
\end{align}
is well-defined and surjective.   The map (\ref{Xunr}) determines a bijective map
\[
X_{\unr}(M) / W^{\fs} \longrightarrow \widehat{\fA}
\]
that is to say,  a bijective map
\[
\mathbf{Kb} \longrightarrow \widehat{\fA}.
\]  

As in the proof of Theorem \ref{spectrum}, the $C^*$-Plancherel Theorem  now secures a homeomorphism
\[
\widehat{\fA} \cong \mathbf{Kb}.
\]
\end{proof} 

The Klein bottle $\mathbf{Kb}$ therefore appears as a  connected compact Hausdorff subspace of the tempered dual of $G$.


\smallskip  

\begin{thebibliography}{9999999}
\bibitem[AFMO]{AFMO1} J.D. Adler, J. Fintzen, M. Mishra, and K. Ohara, ``Structure of Hecke algebras arising from types", preprint arXiv:2408.07801, 2024.
%\bibitem[AFMO2]{AFMO2} \bysame, ``Reduction to depth zero for tame $p$-adic groups via Hecke algebra isomorphisms", preprint arXiv:2408.07805, 2024.
\bibitem[AFO]{AFO} J.D. Adler, J. Fintzen, and K. Ohara, ``A depth-zero principal-series block whose Hecke algebra has a non-trivial cocycle", preprint arXiv:2510.07391, 2025.
%\bibitem[ABPS]{ABPS} A.-M. Aubert, P. Baum, R.J. Plymen, M. Solleveld,  Hecke algebras for inner forms of $p$-adic special linear groups, J. Inst. Math. Jussieu \textbf{16} (2017),  35--419.
\bibitem[Ber]{Ber} J. N. Bernstein, ``Le ``centre" de Bernstein", in \textit{Representations of reductive groups over a local field} (P. Deligne, ed.), Hermann, Paris, 1984, 1--32. 
\bibitem[BK]{BK} C.J. Bushnell and P.C. Kutzko, ``Smooth representations of reductive $p$-adic groups: structure theory via types", Proc. London Math. Soc. (3) \textbf{77} (1998), no. 3, 582--634.
\bibitem[D]{D} J. Dixmier, ``$C^*$-algebras", North-Holland 1982.   
\bibitem[GR1]{GR1} D. Goldberg and A. Roche, ``Types in $\SL_n$", Proc. London Math. Soc. (3) \textbf{85} (2002), 119--138.
\bibitem[GR2]{GR2} \bysame, ``Hecke algebras and $\SL_n$-types", Proc. London Math. Soc. (3) \textbf{90} (2005), 87--131.
%\bibitem[KY]{Kim-Yu} J.-L. Kim and  J.-K. Yu,``Construction of tame types", in \emph{Representation theory, number theory, and invariant theory}, 337--357, Progr. Math. \textbf{323}, Birkhäuser/Springer, Cham, 2017.
%\bibitem[L]{L} R.L.  Lipsman, Group representations, LNM 388 (1974).   
\bibitem[Mo]{Mo} L. Morris, ``Level Zero \textbf{G}-Types", Compositio Math. \textbf{118} (1999), 135--157. 
\bibitem[P]{P} R.J. Plymen, ``Reduced $C^*$-algebra of reductive $p$-adic groups", J. Functional Analysis, \text{88} ( 1990) 251--266.
%\bibitem[R0]{R0} A. Roche, Types and Hecke algebras for principal series representations of split reductive p-adic groups,
%Annales scientifiques de l'\'{E}cole Normale Sup\'{e}rieure (4) 31 (1998), 36--413.  
\bibitem[R]{R} A. Roche, ``Parabolic induction and the Bernstein decomposition", Compositio Math. \textbf{134} (2002), 113--133. 
%\bibitem[S]{S} H. Sato, Algebraic topology: an intuitive approach,  AMS Translations of Math. Monographs 183 (1999). 
\bibitem[W]{W} C.~A.~Weibel, \emph{An Introduction to Homological Algebra}, Cambridge Studies in Advanced Mathematics \textbf{38}, Cambridge University Press, 1994.
\end{thebibliography}
\end{document}